\documentclass[final,leqno]{siamltex}
\usepackage{chngcntr}

\usepackage[arrow, matrix, curve]{xy}
\usepackage[latin1]{inputenc} 	
\usepackage[T1]{fontenc}
\usepackage[english]{babel}
\usepackage{lmodern}
\usepackage{amsmath}
\usepackage{amssymb}  
\usepackage{mathrsfs}
\usepackage{euscript}
\usepackage{enumerate}
\usepackage{xspace}
\usepackage{framed}
\usepackage{cite}
\usepackage{fancyhdr}
\usepackage{bbm}
\usepackage{float}
\usepackage{comment}
\usepackage{textcomp}
\usepackage{url}
\usepackage{srcltx} 
\usepackage{hyperref} 

\DeclareFontShape{OMX}{cmex}{m}{n}{
  <-7.5> cmex7
  <7.5-8.5> cmex8
  <8.5-9.5> cmex9
  <9.5-> cmex10
}{}
\SetSymbolFont{largesymbols}{normal}{OMX}{cmex}{m}{n}
\SetSymbolFont{largesymbols}{bold}  {OMX}{cmex}{m}{n}

\hypersetup{draft=false, colorlinks=true, linkcolor=black, urlcolor=blue, citecolor=black}
\newtheorem{remark}[theorem]{Remark}

\newtheorem{example}[theorem]{Example}

\newcommand{\dd}{\, \mathrm{d}}
\newcommand{\RR}{\mathbb{R}}
\newcommand{\CC}{\mathbb{C}}
\newcommand{\NN}{\mathbb{N}}
\newcommand{\XX}{\mathbb{X}}
\newcommand{\YY}{\mathbb{Y}}

\newcommand{\X}{\mathcal{X}}
\newcommand{\Y}{\mathcal{Y}}
\newcommand{\J}{\mathcal{J}}
\newcommand{\LpmX}{L^p (\mu;\mathcal{X})}
\newcommand{\LpTX}{L^p (0,T;\mathcal{X})}

\author{Martin Burger\thanks{Computational Imaging Group and Helmholtz Imaging, Deutsches Elektronen Synchrotron (DESY) and Fachbereich Mathematik, Universit\"at of Hamburg, Germany ({\tt martin.burger@desy.de})} \and Thomas Schuster\thanks{Department of Mathematics, Saarland University, Saarbr\"ucken, Germany ({\tt thomas.schuster@num.uni-sb.de}).} \and Anne Wald\thanks{Institute for Numerical and Applied Mathematics, University of G\"ottingen, Germany ({\tt a.wald@math.uni-goettingen.de}).}}

\title{Ill-posedness of time-dependent inverse problems in Lebesgue-Bochner spaces}
        
\begin{document}

\maketitle

\begin{abstract}
We consider time-dependent inverse problems in a mathematical setting using Lebesgue-Bochner spaces. Such problems arise when one aims to recover parameters
from given observations where the parameters or the data depend on time. There are various important applications being subject of current research that belong
to this class of problems. Typically inverse problems are ill-posed in the sense that already small noise in the data causes tremendous errors in the solution.
In this article we present two different concepts of ill-posedness: temporally (pointwise) ill-posedness and uniform ill-posedness with respect to the Lebesgue-Bochner setting. We investigate the two concepts by means of a typical setting consisting of a time-depending observation operator composed by a compact operator.
Furthermore we develop regularization methods that are adapted to the respective class of ill-posedness.
\end{abstract}



\section{Introduction}


Time-dependent inverse problems yield a large and versatile subfield of inverse problems, which has seen a growing interest in recent years. Considering an inverse problem, given by an operator equation 
\begin{displaymath}
 F(\vartheta) = y, \quad F: \mathcal{D}(F) \subseteq \mathbb{X} \rightarrow \mathbb{Y},
\end{displaymath}
the dependence on time may be reflected in the source $\vartheta \in \mathbb{X}$, in the data $y \in \mathbb{Y}$, or in the forward operator $F$. The dependence results for example from changes in the underlying physical setup or from an intrinsic time-dependence of the respective system, which evolves in time.

We shall see that typically the operator $F$ in dynamic inverse problems can be decomposed into two types of operators that preserve the causality. In order to compute the output $y(t)$ at a fixed time $t$ we 
apply operators of the form $A_t$ acting at fixed time or $B_{[0,t]}$ acting in a time interval. Again due to causality the output at time $t$ only depends on the restriction $f|_{[0,t]}$ to the initial time interval. We thus see  that we can approach such problems in a global or local setting.

Many time-dependent problems have been addressed and studied in relation to applications (see \cite{book_TDPIIP}), for instance, in dynamic computerized tomography \cite{C1-Hahn_estimation, C1-hahn_local}, magnetic resonance imaging \cite{Hammernik2018}, emmission tomography \cite{hashimoto2019dynamic}, magnetic particle imaging \cite{knopp2012magnetic, A1-KALTENBACHER;ET;AL:21, bathke2017improved, albers2023timedependent}, or structural health monitoring \cite{A1-KLEIN;ET;AL:21, C2-LAMBWAVES-BUCH:18}.  We state some recent examples for dynamic inverse problems in real-world applications in the following: 

\vspace*{2ex}

(a) \emph{Dynamic Computerized Tomography (DCT)}\\
The aim of CT is to recover the interior $\vartheta$ of an object from X-ray measurements. If the object undergoes a motion (e.g. if the patient is moving), then
the function $\vartheta$ to be recovered depends on time and the forward operator is given as
\begin{equation}\label{DCT} 
 F [\vartheta (t)] = S(t) R [\vartheta (t)].  
\end{equation}
Here $R$ denotes the 2D Radon transform acting on the spatial variable of $\vartheta (t,x)$ and $S(t)$ describes the measurement geometry. The observation operator $S(t)$
depends on $t$, since in this way, changes in the measurement process in time are also included in the mathematical model. In the notation above we have
\[   y(t) = A_t [ \vartheta ],   \] 
the problem is fully local in time.

In particular, we want to point out two approaches: The first one is to reconstruct the initial state of the object from data that has been collected while the object is changing in time, and taking into account the motion of the object (see, e.g., \cite{bh2014, bh2017}). In this case the motion of the object itself is not of interest. The second approach is to reconstruct the moving object, i.e., the object is to be reconstructed at a range of time points. Here, one usually has to deal with sparse data or limited angle problems \cite{bubba2017, burger_et_al_2017, Arridge22DIP}. Including motion, the Radon transform is now computed for a deformed version of $\vartheta$ over the whole time interval, which changes the setting to
\[   y(t) = A_t B_{[0,t]} [\vartheta ].  \]

(b) \emph{Dynamic Load Monitoring (DLM)}\\
DLM aims for computing loads in elastic structures from time-dependent sensor data that are acquired at the structure's surface. If the displacements $u(t,x)$ are small, then 
the wave propagation in an elastic material is described using Hooke's law by the Cauchy equation of motion
\begin{equation}\label{elastic-WE}
\rho \ddot{u}(t,x) - \nabla\cdot (\mathbb{C}:\varepsilon (u))(t,x) = f(t,x).
\end{equation}
Here $\mathbb{C}(x)$ denotes the elasticity tensor in $x$, $\rho$ is the mass density, $\varepsilon (u) = (\nabla u + \nabla u^\top)/2$ is the linearized Green strain tensor,
and $f(t,x)$ is a volume body force. The hyperbolic PDE system \eqref{elastic-WE} is uniquely
solvable if appropriate initial and boundary values are given. \emph{Dynamic load monitoring}
means computing $f$ (the load) from (partial) measurements of displacements $u$ at the boundary of a structure. 
This is a linear inverse problem with forward operator
\[   F(f) := S(t) [L f ],   \]
and the operator $L$ maps a source term $f$ to the (unique, weak) solution $u =: Lf$ of \eqref{elastic-WE} equipped with appropriate initial and boundary values, whereas $S(t)$ is the mathematical model of the data acquisition process. E.g., $S(t)$ equals the the trace operator $\gamma$,
\[  \gamma [g] = g_{|_{\Gamma}},   \]
with $\Gamma\subset \partial\Omega$
or 
\[   S(t) [g] = \Big( \int_\Gamma \langle g(x), \chi_j (x)\rangle \,\mathrm{d} s_x \Big)_{j=1,\ldots,J},  \]
where $\chi_j$ is the characteristic function of sensor $j$.
Note that in these examples $S$ is actually independent of time, but can be extended (by identical copies) as an operator on a time-dependent $g$ defined on $[0,T]\times \partial \Omega$, yielding a trace on $\Gamma_T = [0,T]\times \Gamma$. Let us mention that $L$ depends on the whole time interval, hence we effectively compute the data via 
\[   y(t) = A_t [B_{[0,t]} f ].   \]
  


(c) \emph{Magnetic Particle Imaging (MPI)}\\
MPI is a relatively novel medical imaging technique, e.g., to monitor blood vessels in a patient. Magnetic nanoparticles are injected into the bloodstream and distribute inside the patient's body. A strong external magnetic field with a field free point (FFP) is applied such that the nanoparticles' magnetic moment vectors align with the field lines of this field. Only in the field free point, where the applied field is very weak, the particles can move freely. When the FFP moves around the field of view, these particles inside the FFP abruptly change their magnetization. The resulting change in the total magnetic field induces voltages in the receive coils at each time point $t$ during the measurement, which serve as data (see also \cite{knopp2012magnetic}).\\
The goal in MPI is to reconstruct the concentration $c$ of the magnetic particles inside the body from the measured voltages. The physical model is given by the $k=1,\ldots,K$ integral equations
\begin{displaymath}
 v_k(t) = \int_0^t \tilde{a}_k(t-\tau) \int_{\Omega} c(x) s_k(x,\tau) \dd x \ \dd \tau, \quad s_k(x,t) = p_k^\mathrm{R}(x)^T \partial_t \overline{m}(x,t),
\end{displaymath}
where $K$ is the number of receive coils, $\tilde{a}_k$, $k=1,\ldots,K$, are transfer functions, $p^R_k$ the coil sensitivity of the $k$-th coil, and $\overline{m}$ is the mean magnetization. The function $s_k$ is called the system function and represents a kind of induction potential for the $k$-th coil. In practice, the system function is measured in a time-consuming calibration process. This yields two types of inverse problems. First, the actual imaging problem is to reconstruct the concentration $c$ from the time-dependent data. Secondly, to avoid the expensive calibration, the model-based reconstruction of the system function from data that was obtained for known particle concentrations yields an inverse problem where both the source and the data are time-depending. 
In both cases, the forward operator can be formulated as a composition 
\begin{displaymath}
 F[f](t) = S  I(t)[f(t)],
\end{displaymath}
where $S$ represents the convolution in time and $I(t)$ the integral operator in space given a time-varying weight. Note that Lebesgue-Bochner spaces are a convenient choice for the source space of the inverse problem for the system function. The data space is given by the Lebesgue space $L^s([0,T])$.
In this case we can write  the problem in the form
 \[   y(t) =B_{[0,t]} A_t [f ]. \]  

\vspace*{2ex}

The concept of ill-posedness for inverse problems has been discussed in detail in the literature. Linear ill-posed problems have been classified by Hadamard \cite{hadamard1923}.
For nonlinear forward operators, the definition has to be adapted to the local character of nonlinear operators, which has been done by Hofmann and Scherzer in \cite{hs94, hs98, bh00}. \\
Here, we address both linear and nonlinear problems, i.e., our definition of ill-posedness will be based on the definition by Hofmann and Scherzer.

\vspace*{2ex}

{\color{black}
Since the dependence on time of a quantity is of an entirely different nature than the dependence on space, it is not surprising that this difference is also reflected in the respective mathematical modelling. Regarding inverse problems, the choice of Lebesgue-Bochner spaces as source or data space yields an adequate setting, since they admit a different treatment of these two types of variables. A prominent example where this setting occurs naturally is parameter identification for parabolic linear partial differential equations. The differing roles of the temporal and spatial variables in this context are addressed in \cite[Ch.~23.1]{zeidler2013linear} for a classical initial boundary value problem for the heat equation,
\begin{align*}
 \partial_t u - \Delta u &= f &&\text{in } \Omega \times (0,T), \\
                       u &= 0 &&\text{on } \partial\Omega \times [0,T], \\
                  u(x,0) &= u_0(x) && \text{in } \Omega.
\end{align*}
We summarize some of the properties from \cite[Ch.~23.1]{zeidler2013linear}:
\begin{itemize}
 \item[(a)] For a fixed time $t$, the function $x \mapsto u(x,t)$ is an element of a Sobolev space $V$ and we denote this function by $u(t)$.
 \item[(b)] If we now vary $t \in [0,T]$, we obtain the function $t \mapsto u(t)$. In this way, we obtain a function with values in a Banach space $V$, in contrast to the real-valued function $(x,t) \mapsto u(x,t)$.
\end{itemize}
Multiplication with a test function $v \in V$ and integration over $\Omega$ yields 
\begin{equation} \label{ex_variationalformulation}
 \frac{\mathrm{d}}{\mathrm{d}t} \left( u(t), v \right)_H + a\left(u(t),v\right) = \left( f(t), v \right)_H \quad \text{in } (0,T) \text{ for all } v\in V, \quad u(0) = u_0 \in H 
\end{equation}
with
\begin{displaymath}
 \begin{split}
  a\left(w,v\right) &:= \int_{\Omega} \sum_{i=1}^N \partial_{x_i} w_i(x) \partial_{x_i} v(x) \mathrm{d}\, x, \\
  \left(f(t),v\right)_H &:= \int_{\Omega} f(x,t) v(x) \mathrm{d}\, x.
 \end{split}
\end{displaymath}
This variational formulation shows the necessity to use two function spaces $V$ and $H$:
\begin{itemize}
 \item[(c)] The time-derivative yields the space $H$, whereas $V$ is obtained from the spatial derivative $-\Delta$ and the boundary condition. In this example for the heat equation, we have $V \subseteq H$ and $V$ is dense in $H$. In this context, we use the \emph{Gelfand} or \emph{evolution triple}
 \begin{displaymath}
  V \subseteq H \subseteq V^*.
 \end{displaymath}
 \item[(d)] The time-derivative $\tfrac{\mathrm{d}}{\mathrm{d}t}$ is interpreted as a generalized derivative, i.e., the variational equation \eqref{ex_variationalformulation} has to be satisfied only for almost all $t \in (0,T)$.
 \item[(e)] These considerations result in the choice
   \begin{displaymath}
    u \in W^1_2(0,T;V,H) = H^{1}(0,T;V,H)
   \end{displaymath}
   as the state space for \eqref{ex_variationalformulation}.
\end{itemize}

Mathematical frameworks, particularly suited for parameter identification for partial differential equations, have been presented in \cite{kaltenbacher17, akar16, Hoffmann_2022}. In \cite{kaltenbacher17, akar16}, the focus is on time-depentent PDEs and Lebesgue-Bochner spaces are used to formulate the respective mathematical problems.


}

\emph{Outline.} Section 2 provides the reader with all necessary preliminaries about Lebesgue-Bochner spaces. We particularly discuss criteria for relatively compact subsets in such spaces. In Section 3 we develop the aforementioned concepts for pointwise and uniform ill-posedness of time-dependent inverse problems. As an example we consider time-dependent observations of a compact operator which applies, e.g., to dynamic CT. Corresponding to these concepts, we construct two types of regularization concepts in Section 4. For both methodologies we present examples: variational tracking, classical variational techniques and Kaczmarz-based regularization. The article ends with an outlook to future research in the field.


\section{Some preliminaries about Lebesgue-Bochner spaces}


We start with a short introduction to the theory of Lebesgue-Bochner spaces. We use the notation $\mathbb{K} \in \lbrace \RR,\CC \rbrace$ and let $1 \leq p, p^{*} < \infty$ with $\frac{1}{p}+\frac{1}{p^*}=1$.
If not otherwise specified, let $\mathcal{X}$ be a Banach space with norm $\lVert \cdot \rVert_{\mathcal{X}}$. Its dual is denoted by $\mathcal{X}^*$.

\vspace*{2ex}

Recall that a function $u : \Omega \rightarrow \mathcal{X}$ is called \emph{strongly $\mu$-measurable}, if there is a sequence $\lbrace u_n : \Omega \rightarrow \mathcal{X} \rbrace$ of simple functions such that $u$ is its pointwise limit, i.e., $\lim_{n\rightarrow \infty} u_n(t) = u(t)$ for $t \in \Omega$ a.e.

\vspace*{2ex}

For a finite measure space $(\Omega,\mu)$ with a measure $\mu$ we define the space
\begin{displaymath}
 L^p(\mu; \mathcal{X}) := \left\lbrace u : \Omega \rightarrow \mathcal{X} \, | \, u \text{ is strongly $\mu$-measurable and } \lVert u \rVert_{p,\mathcal{X}} < \infty  \right\rbrace,
\end{displaymath}
where
\begin{equation}
 \lVert u \rVert_{p,\mathcal{X}} := \left( \int_{\Omega} \lVert u \rVert_{\mathcal{X}}^p \dd \mu \right)^{\frac{1}{p}}.
\end{equation}
Note that we identify functions that are equal $\mu$-almost everywhere in $[0,T]$.
The space $L^p(\mu; \mathcal{X})$ is called a \emph{Lebesgue-Bochner space} or just \emph{Bochner space}.
If endowed with the norm $\lVert \cdot \rVert_{p,\mathcal{X}}$, the space $L^p(\mu; \mathcal{X})$ is a Banach space.
For the special case $\mathcal{X}=\mathbb{R}$ we briefly write $L^p (\mu)$.

If $\mathcal{X}$ is a Hilbert space and $p=2$, then $L^2(\mu; \mathcal{X})$ is also a Hilbert space with inner product
\begin{displaymath}
 \left( u_1, u_2 \right)_{L^2(\mu; \mathcal{X})} = \int_{\Omega} (u_1, u_2)_{\mathcal{X}} \dd \mu.
\end{displaymath}

Particularly in applications, the measure space is often given as the time interval $\Omega = [0,T] \subseteq \RR$, $T> 0$, with the Lebesgue measure. The respective Bochner spaces
\begin{equation}
 L^p(0,T; \mathcal{X}) := \left\lbrace u : (0,T) \rightarrow \mathcal{X} \, | \, u \text{ is measurable and } \lVert u \rVert_{p,\mathcal{X}} < \infty  \right\rbrace
\end{equation}
play an important role in the theory of nonlinear operator equations or evolution equations. The following statement is taken from \cite{zeidler2013linear}.

\vspace*{2ex}

\begin{proposition}
 Let $\widetilde{\mathcal{X}}$ be a Banach space over $\mathbb{K}$.
 \begin{itemize}
  \item[(i)] The space $C^m([0,T]; \mathcal{X})$, $m=0,1,...$, of all continuous functions $u: [0,T] \rightarrow \mathcal{X}$ with continuous derivatives up to order $m$ on $[0,T]$ with the norm 
  \begin{displaymath}
   \lVert u \rVert := \sum_{i=0}^m \max_{t \in [0,T]} \lvert D^{(i)}u(t) \rvert,
  \end{displaymath}
 where $D^{(i)}$ denotes the $i$-th derivative, is a Banach space over $\mathbb{K}$.
  \item[(ii)] The set of all step functions $\sigma: [0,T] \rightarrow \mathcal{X}$ is dense in $L^p(0,T; \mathcal{X})$.
  \item[(iii)] The space $C([0,T]; \mathcal{X})$ is dense in $L^p(0,T; \mathcal{X})$ and the respective embedding is continuous.
  \item[(iv)] $L^p(0,T; \mathcal{X})$ is \begin{itemize}
                                          \item separable, if $\mathcal{X}$ is separable,
                                          \item uniformly convex, if $1 < p < \infty$ and $\mathcal{X}$ is uniformly convex.
                                         \end{itemize}
  \item[(v)] If $X \hookrightarrow \widetilde{\mathcal{X}}$ is continuous, then
   \begin{displaymath}
    L^q(0,T; \mathcal{X}) \hookrightarrow L^r(0,T; \widetilde{\mathcal{X}}), \quad 1 \leq q \leq r \leq \infty,
   \end{displaymath}
   is also continuous.
 \end{itemize}
\end{proposition}

\vspace*{2ex}

The H\"older inequality transfers from Lebesgue to Bochner spaces and is used to derive the dual space of a Bochner space, see, e.g., \cite{zeidler2013linear}. 

\vspace*{2ex}

\begin{proposition}
 For $1 < p < \infty$ and $u \in L^p(0,T; \mathcal{X})$, $v \in L^{p^*}(0,T; \mathcal{X}^*)$, the H\"older inequality
 \begin{displaymath}
  \int_{0}^T \left\lvert \big\langle v(t), u(t) \big\rangle_{\mathcal{X}^* \times \mathcal{X}} \right\rvert \dd t \leq \left( \int_0^T \lVert v(t) \rVert_{\mathcal{X}^*}^{p^*} \dd t \right)  \cdot \left( \int_0^T \lVert u(t) \rVert_{\mathcal{X}}^{p} \dd t \right)
 \end{displaymath}
 holds and we have
 \begin{displaymath}
  \big( L^p(0,T; \mathcal{X}) \big)^* = L^{p^*}(0,T; \mathcal{X}^*).
 \end{displaymath}
\end{proposition}

Regarding ill-posed problems, compact operators are of special interest since their inverse is only continuous if the operator has finite dimensional range.
This is why the characterization of relatively compact sets in Lebesgue-Bochner spaces is very important. 
Recall that an operator $T: \mathcal{X} \rightarrow \mathcal{Y}$ between Banach spaces is compact, if it maps compact subsets in $\mathcal{X}$ to relatively compact subsets of $\mathcal{Y}$. Relatively compact subsets in Lebesgue-Bochner spaces have been characterized by Diaz and Mayoral \cite{diaz1999compactness} as well as recently in the two articles \cite{jvn07, jvn14} by van Nerveen. To present the main result from \cite{jvn14} we need some further definitions.

A set $V\subset \LpmX$ is called \emph{uniformly $L^p$-integrable} if
\begin{equation}\label{eq-uniformly-lp-integrable}  \lim_{r\to\infty} \sup_{f\in V} \|\chi_{\|f(t)\| >r} f\|_p = 0 , \end{equation}
where $\chi_{\|f(t)\| >r}$ is the characteristic function of the set $\{t :\|f(t)\|>r\}$. The set $V$ is called \emph{uniformly tight},
if for all $\varepsilon>0$ there exists a compact set $\mathcal{K}\subset \mathcal{X}$ such that
\begin{equation}\label{eq-uniformly-tight} \sup_{f\in V} \mu (\{ t : f(t)\not\in \mathcal{K}\})\leq \varepsilon,  \end{equation}
and it is called \emph{scalarly relatively compact} if for all $x^*\in \mathcal{X}^*$ the set $\{ t\mapsto \langle f(t),x^* \rangle : f\in V\}$
is relatively compact in $L^p (\mu)$.

\vspace*{1ex}

\begin{proposition}[{\cite[Theorem 1]{jvn14}}]\label{P-set-rel-comp}
Let $1\leq p < \infty$. A subset $V\subset \LpmX$ is relatively compact if and only if it is uniformly $L^p$-integrable, uniformly tight, and scalarly relatively compact.
\end{proposition}

\vspace*{1ex}

Bounded sequences $\{f_n\}\subset \LpmX$ fulfill the first two conditions in Proposition \ref{P-set-rel-comp}:

\vspace*{1ex}

\begin{lemma}\label{L-fn-int-tight}
If $\{f_n\}\subset \LpmX$ is a bounded sequence, then $\{f_n\}$ is uniformly $L^p$-integrable and uniformly tight.
\end{lemma}

\vspace*{1ex}

\begin{proof}
Let $c_f>0$ be a constant with $\|f_n\|_{\LpmX} \leq c_f$ for all $n\in\NN$.\\
(a) We prove that
\begin{equation}\label{mu-f-n-0}
\lim_{r\to\infty} \sup_{n\in\NN} \mu \big( \{ t :  \|f_n (t)\|_{\X} > r \} \big) = 0.
\end{equation}
We assume the contrary. Let $\{r_k\}_{k\in\NN}$ be a sequence with $\lim_{k\to\infty} r_k = +\infty$ such that
\[
\sup_{n\in\NN} \mu \big( \{ t :  \|f_n (t)\|_{\X} > r_k \} \big) > R > 0 \ \text{for all } k\in\NN
\]
with some constant $R>0$. Then for each $k\in\NN$ there exists an index $n_k\in\NN$ with
\[
\mu \big( \{ t :  \|f_{n_k} (t)\|_{\X} > r_k \} \big) > R.
\]
This implies the existence of values $t_1 (k), t_2 (k)\in [0,T]$ such that $\mu( [t_1 (k),t_2 (k)])>R$ and
\[
\int^{t_2 (k)}_{t_1 (k)} \| f_{n_k} (t)\|_{\X}\, \dd \mu(t) > \mu([t_1 (k),t_2(k)])r_k > R r_k.
\] 
But this immediately leads to
\[
\lim_{k\to\infty} \|f_{n_k}\|_{\LpmX} = \lim_{k\to\infty} \int_0^T \| f_{n_k} (t)\|_{\X}\, \dd \mu(t) > R r_k = +\infty,
\] 
which contradicts the boundedness of $\{f_n\}$. Combining $\|f_n\|\leq c_f$ and \eqref{mu-f-n-0} we finally obtain
\begin{eqnarray*}
\lim_{r\to\infty} \sup_{n\in\NN} \big\| \chi_{\|f_n\|>r} f_n \big\|_{\LpmX}^p &=&
\lim_{r\to\infty} \sup_{n\in\NN} \int_0^T \big\| \chi_{\|f_n\|>r} f_n (t) \big\|_{\X}^p\, \dd \mu(t)\\ 
&\leq & \lim_{r\to\infty} \sup_{n\in\NN} \mu \big( \{ t :  \|f_n (t)\|_{\X} > r \} \big) \sup_{n\in\NN} \|f_n\|_{\LpmX}^p = 0,
\end{eqnarray*}
proving that $\{f_n\}$ is uniformly $L^p$-integrable.\\
(b) Next we prove that $\{f_n\}$ is uniformly tight. To this end we use arguments similar to those in the proof of Theorem 1 from \cite{jvn14}.
Let $\varepsilon>0$ be given arbitrarily. Fix $n\in\NN$. Since $f_n$ is strongly measurable, its distribution on the Borel sets $\mathcal{B}(\X)$
is tight. According to Ulam's Theorem (see, e.g. \cite[Th. 3.1]{VAKHANIA;ET;AL:87}) it follows that to each $n$ there exists a compact set $\mathcal{K}_n\subseteq \X$ such that
\begin{equation}\label{mu_fn_2n}
\mu \big( \{ t: f_n (t)\not\in \mathcal{K}_n\} \big) \leq 2^{-n}.
\end{equation}
Let $n_0\in\NN$ be sufficiently large that $2^{1-n_0}<\varepsilon$. Define further the sets 
$$L_n:= \{ x\in\X : d(x,\mathcal{K}_n)<2^{-n} \}$$
and
\[
\mathcal{K}:= \overline{\bigcap_{n\geq n_0} L_n} \subseteq \X.
\]
We prove that $\mathcal{K}$ is totally bounded. Since $\mathcal{K}_n$ is compact, there are finitely many open balls $B(x_i,2^{-n})$, $x_i\in\X$, that cover $\mathcal{K}_n$.
Then the finitely many open balls $B(x_i,2^{1-n})$ cover $\overline{L_n}$. Hence, $\mathcal{K}$ is totally bounded and as a closed set also compact.
Furthermore we have by \eqref{mu_fn_2n}
\begin{eqnarray*}
\mu \big( \{ t: f_n (t)\not\in \mathcal{K} \} \big) & \leq & \sum_{n\geq n_0} \mu \big( \{ t: f_n (t)\not\in L_n \} \big) 
\leq  \sum_{n\geq n_0} \mu \big( \{ t: f_n (t)\not\in \mathcal{K}_n \} \big)\\
& \leq & \sum_{n\geq n_0} 2^{-n} \leq 2^{1-n_0} < \varepsilon.
\end{eqnarray*}
Since $n\in\NN$ is arbitrary we have
\[  \sup_{n\in\NN} \mu \big( \{ t: f_n (t)\not\in \mathcal{K} \} \big) < \varepsilon  \]
showing that $\{f_n\}_{n\in\NN}$ is uniformly tight.
\end{proof}

\vspace*{1ex}

Lemma \ref{L-fn-int-tight} is important especially to prove compactness of operators between spaces as $\LpmX$. Another characterization of relatively compact sets in $\LpmX$, where $\mu$ is the Lebesgue measure, is given in \cite{SIMON:86, GAHN;RADU:16}. For $z\in (0,T)$ and $f\in \LpmX$ we define the translation operator 
\[   \tau_z : [-z,T-z] \to \X,\qquad (\tau_z f)(t) := f(t+z) .   \]

\begin{theorem}[{\cite[Theorem 1]{SIMON:86}}]\label{T-Gahn-Radu}
Let $p\in [1,+\infty )$ and $V\subset \LpTX$. Then $V$ is relatively compact, if and only if
\begin{itemize}
\item[(a)] for any subset $(a,b)\subset [0,T]$ the set $\big\{ \int_a^b f(t)\,\mathrm{d} (t) : f\in V\big\}$
is relatively compact in $\X$,
\item[(b)] for $z$ with $0\leq z < T$ it holds
\[  \sup_{f\in V} \|\tau_z f - f\|_{L^p ([-z,T-z],\X)} = 0\qquad \mbox{for } z\to 0.  \]
\end{itemize}
\end{theorem}

Theorem \ref{T-Gahn-Radu} is an analogue of the well-known Arzel\`{a}-Ascoli Theorem, which characterizes relatively compact sets in $\mathcal{C}([0,T],\mathcal{X})$. Note that $p=+\infty$ is not possible: For example, the set $\{f\}$ for a discontinuous $f\in L^\infty (0,T;\Omega)$ is compact, but condition (b) of Theorem \ref{T-Gahn-Radu} is not satisfied. As a consequence from Lemma \ref{L-fn-int-tight} and Theorem \ref{T-Gahn-Radu} we immediately get

\vspace{1ex}

\begin{corollary}
A bounded sequence $\{f_n\}\subset \LpTX$ is relatively compact in $L^p (0,T;\X)$, if
for $z$ with $0\leq z < T$ and $x^*\in\X^*$ it holds
\begin{equation}\label{sup_n_g}  \sup_{n\in\NN} \|\tau_z g_n - g_n\|_{L^p ([-z,T-z])} = 0\qquad \mbox{for } z\to 0,  
\end{equation}
where $g_n (t) := \langle f_n (t),x^* \rangle$.
\end{corollary}

\vspace*{1ex}

\begin{proof}
Regarding Lemma \ref{L-fn-int-tight} we have to show that for any $x^*\in\X^*$ the sequence $\{g_n\}$ is relatively
compact in $L^p ([0,T])$. Since by assumption $\{f_n\}$ is bounded
in $L^p (0,T,\X)$, i.e. $\|f_n\|\leq c_f$, we have for any $(a,b)\subset [0,T]$
\[  \Big| \int\limits_a^b g_n (t)\, \mathrm{d} t \Big| \leq \int\limits_a^b \|f_n (t)\|_{\X} \|x^*\|_{\X^*}\,\mathrm{d} t
    \leq c_f \|x^*\|_{\X^*},   \]
proving that $\big\{ \int\limits_a^b g_n (t)\, \mathrm{d} t : n\in\NN \big\}$ is bounded and hence relatively compact in $\RR$.
Condition \eqref{sup_n_g} is assumption (b) in Theorem \ref{T-Gahn-Radu} applied to $\{f_n\}$, which completes the proof.
\end{proof}

\vspace*{1ex}

It is difficult to verify condition \eqref{sup_n_g} for specific settings. Since parameter identification problems for partial differential equations play a crucial role in view of many real-world applications, compact embedding theorems for Sobolev spaces are of utmost importance. For completeness we recapitulate some of them to finish this section.\\

\begin{theorem}[The Rellich-Kondrachov theorem]
The Sobolev space embedding $W^{m,p} (\Omega) \hookrightarrow L^q (\Omega)$, $\Omega$ a bounded domain in $\RR^n$, is compact,
if any of these conditions holds true:
\begin{itemize}
\item[(a)] $n-mp>0$ and $1\leq q < np/(n-mp)$,
\item[(a)] $n=mp$ and $1\leq q < +\infty$,
\item[(c)] $mp>n$ and $1\leq q \leq +\infty$.\\
\end{itemize}
\end{theorem}


\section{Ill-posedness in Lebesgue-Bochner spaces}


We at first introduce a setting in Lebesgue-Bochner spaces for time-dependent inverse problems. Consider Banach spaces $\X$, $\Y$, and an operator equation
\begin{equation} \label{ip}
 F(\vartheta) = y, \quad F: \mathcal{D}(F) \subseteq \mathbb{X} \rightarrow \mathbb{Y},
\end{equation}
where we set
\begin{displaymath}
 \mathbb{X} := L^p(0,T; \mathcal{X}), \quad \mathbb{Y} := L^{q}(0,T; \mathcal{Y})
\end{displaymath}
for any $1\leq p,q < \infty$. We assume that only noisy data, denoted by $y^{\delta}$, are available with a noise level $\delta > 0$, fulfilling
\begin{displaymath}
 \left\lVert y - y^{\delta} \right\rVert_{\mathbb{Y}} \leq \delta.
\end{displaymath}

This is a fairly general mathematical model for time-depending problems, since in many applications, the Banach space $\mathcal{X}$ is a suitable function space such as $L^r(D)$ or $H^1(D)$ on some open subset $D \subseteq \RR^N$, $N \in \NN$. As an example, weak solutions of hyperbolic and parabolic equations often are elements of spaces such as $L^p(0,T;H^1 (D))$ or subspaces thereof.
It is important to emphasize the difference in the physical meaning of the temporal variable $t \in [0,T]$ and spacial variables $x=(x_1,...,x_N)^T \in D$, on which both the source function $\vartheta$ as well as the data $y$ may depend. The principle of \emph{causality} describes the nature of time in contrast to space: it is usually possible to move freely in space, but one can only advance in time. In the context of evolution equations, which describe the evolution of a system in time, Bochner spaces allow us to encode the specific role of time: At a fixed time $t_0$, the system is described by a function $u(t_0) \in \mathcal{X}$ where we can encode, e.g., regularity properties with respect to the spacial variable.

\vspace*{1ex}

The question that arises at this point is how the nature of the temporal variable can be reflected in the notion of ill-posedness in a suitable manner. To this end, we introduce two concepts of ill-posedness with respect to time.

\vspace*{1ex}

\begin{definition}
 The inverse problem \eqref{ip} with $F = F_t$ for all $t$ is called \emph{temporally (locally) ill-posed}, if there is a set of positive measure  $\Sigma \subset (0,T)$ such that for almost every $t_0 \in \Sigma$ the operator equation
 \begin{displaymath}
  F_{t_0}(\tilde{\vartheta}) := F(\vartheta(t_0)) = \tilde{y}, \quad F_{t_0} : \mathcal{D}(F_{t_0}) \subseteq \mathcal{X} \rightarrow \mathcal{Y} 
 \end{displaymath}
 is (locally) ill-posed. Here, 
\[  \mathcal{D}(F_{t_0}) := \big\{ \vartheta (t_0)\in \mathcal{X} : \vartheta \in \big(C([0,T];\mathcal{X})\cap \mathcal{D}(F)\big)  \big\}  \]
and $\tilde{y}\in F_{t_0}(\mathcal{D}(F_{t_0}))$. This means, that for all $\rho>0$ there is a sequence $\{\tilde{\vartheta}^{(\rho)}_k\}_{k\in\NN}\subseteq
B_\rho^{\mathcal{X}} (\tilde{\vartheta}^+) \cap \mathcal{D}(F_{t_0})$ with
\[   \lim_{k\to\infty} \|\tilde{\vartheta}^{(\rho)}_k-\tilde{\vartheta}^+\|_{\mathcal{X}} \not= 0,\ \text{but} \ 
     \lim_{k\to\infty} \|F_{t_0}(\tilde{\vartheta}^{(\rho )}_k)-F_{t_0}(\tilde{\vartheta}^+)\|_{\mathcal{Y}} = 0.   \]
\end{definition}

Let us mention that an obvious example for a locally ill-posed problem on Bochner spaces is the dynamically sampled Radon transform discussed in the introduction. If the sampling operator $S(t)$ does not map to a finite dimensional space, the compactness of the Radon transform immediately implies the ill-posedness of the concatenation.

\vspace*{2ex}

\begin{definition}
 The inverse problem \eqref{ip} is called \emph{uniformly (locally) ill-posed} in $\vartheta^+$, if for each $\rho > 0$ there is a sequence $\{ \vartheta^{(\rho)}_k \}_{k \in \NN} \subseteq B_{\rho}(\vartheta^+) \cap \mathcal{D}(F)$ with \begin{displaymath}                                                                                                                                
  \vartheta^{(\rho)}_k \nrightarrow \vartheta^+, \ \text{but} \ F(\vartheta^{(\rho)}_k) \rightarrow F(\vartheta^+)
 \end{displaymath}
 for $k \rightarrow \infty$, i.e.,  
 \begin{displaymath}
  \lim_{k \rightarrow \infty} \left\lVert F(\vartheta^{(\rho)}_k) - F(\vartheta^+) \right\rVert_{\mathbb{Y}} = \lim_{k \rightarrow \infty} \int_0^T \left\lVert F(\vartheta^{(\rho)}_k(t)) - F(\vartheta^+(t)) \right\rVert_{\mathcal{Y}}^q \dd t = 0.
 \end{displaymath}
\end{definition}

\vspace*{2ex}

Note that the statement $\vartheta^{(\rho)}_k \nrightarrow \vartheta^+$ for $k \rightarrow \infty$ translates to
\begin{displaymath}
 \int_0^T \left\lVert \vartheta^{(\rho)}_k (t) - \vartheta^+ (t) \right\rVert_{\mathcal{X}}^p \dd t \nrightarrow 0,
\end{displaymath}
which means that there is a subset $I \subseteq [0,T]$ with Lebesgue measure $\mu(I) > 0$ such that
\begin{equation}
 \int_I \left\lVert \vartheta^{(\rho)}_k (t) - \vartheta^+ (t) \right\rVert_{\mathcal{X}}^p \dd t > 0,
\end{equation}
and thus
\begin{displaymath}
  \left\lVert \vartheta^{(\rho)}_k - \vartheta^+ \right\rVert_{\mathcal{X}} > 0
\end{displaymath}
on $I$.

\vspace*{2ex}

\subsection{An example: time-dependent observations of compact operators on $\mathcal{X}$}

We consider the linear operator equation
\begin{equation} \label{op_eq_ex}
 F(\vartheta) = y,
\end{equation}
where $\vartheta \in \mathbb{X} := L^p(0,T; \mathcal{X})$ and $y \in \mathbb{Y} := L^q(0,T;\mathcal{Y})$ and the spaces $\mathcal{X}$ and $\mathcal{Y}$ are Banach spaces. We assume that the operator $F$ has a representation
\begin{equation}\label{F_bounded_compact}
 F[\vartheta(t)] = S(t)K[\vartheta(t)]
\end{equation}
with a compact linear operator
\begin{displaymath}
 K: \mathcal{X} \rightarrow \mathcal{Y}
\end{displaymath}
and operators $S(t) \in \mathcal{L}(\mathcal{Y})$ that are linear and bounded for every $t\in [0,T]$. We furthermore suppose that the family of operator norms $\{\|S(t)\|\}$ is uniformly bounded,
\begin{equation}\label{S_uni_bounded}
\sup_{t\in (0,T)} \|S(t)\|\leq c_S,
\end{equation}
for a constant $c_S>0$. A setting such as \eqref{F_bounded_compact} is very important regarding practical applications. For instance in dynamic CT, the function $\vartheta\in \mathbb{X}$ represents a moving object, the operator $K$ is the Radon transform, and $S(t)$ models the time-depending measurement process, see \eqref{DCT}.

\vspace*{2ex}

Note that the operator $F_{t_0}: \mathcal{X} \rightarrow \mathcal{Y}, \ F_{t_0}(\vartheta(t_0)) = S(t_0)K[\vartheta(t_0)]$ is compact for each fixed $t_0 \in [0,T]$ and $\vartheta\in C([0,T];\mathcal{X})$ since it is a composition of a compact and a bounded operator. This directly allows us to formulate the following proposition:

\vspace*{2ex}

\begin{proposition}
If $\mathrm{dim}\,(F_{t_0}(\X))=\infty$ for all $t_0\in [0,T]$, then the inverse problem \eqref{op_eq_ex} is temporally ill-posed for all $t_0\in [0,T]$.\\
\end{proposition}

As a specific example we consider the inverse problem of Dynamic Load Monitoring (DLM) and show its uniform ill-posedness 
in an appropriate Lebesgue-Bochner setting. At first we need a decent mathematical setup, which we recapitulate from 
\cite{BINDER;SCHUSTER:15}. We state Cauchy's equation of motion for a linear elastic body $\Omega\subset \mathbb{R}^3$,
\begin{equation}\label{elastic-WE-2}
\rho (x) \ddot{u}(t,x) - \nabla\cdot (\mathbb{C}(x):\varepsilon (u))(t,x) = f(t,x),\qquad (t,x)\in [0,T]\times \Omega
\end{equation}
with the fourth order elasticity tensor $\mathbb{C}\in H^1 (\Omega,\mathbb{R}^{3\times 3\times 3\times 3})$, the mass density $\rho\in  L^{\infty} (\Omega)$, the linearized Green strain tensor $\varepsilon (u) = (\nabla u + \nabla u^\top)/2$, and the vector field $f(t,x)$ representing the dynamic load acting on $x$ at time $t$. Under the additional assumptions that
\begin{equation}\label{rho}   0 < \rho_{\mathrm{min}} < \rho (x) < \rho_{\mathrm{max}} < \infty   \end{equation}
and
\begin{equation}\label{C}   \sup_{x\in \overline{\Omega}}( X,\mathbb{C}(x) : X)_F \geq \alpha \|X\|^2_F\qquad 
     \mbox{for all } X\in \mathbb{R}^{3\times 3}  \end{equation}
with $\alpha > 0$ we get the following existence and uniqueness result for a weak solution of an initial boundary value problem
associated with \eqref{elastic-WE-2}.\\

\begin{proposition}\label{P-solution-ibvp}
Let $\Omega$ be a bounded domain with Lipschitz continuous boundary, $u_0\in H^1(\Omega)^3$, $u_1\in L^2 (\Omega)^3$, $f\in L^2 (0,T; L^2 (\Omega)^3)$ and the
assumptions \eqref{rho}, \eqref{C} hold true. Then, the elastic wave equation \eqref{elastic-WE-2} equipped with the
initial and boundary values 
\begin{eqnarray}
\label{Ncond} [\mathbb{C}(x):\varepsilon (u)]\cdot \nu &=& 0\qquad \mbox{on } [0,T]\times \partial\Omega \\
\label{IC1} u(0,x) &=& u_0 (x)\qquad \mbox{in } \Omega\\
\label{IC2} \dot{u} (0,x) &=& u_1 (x)\qquad \mbox{in } \Omega
\end{eqnarray}
has a unique weak solution $u\in L^2 (0,T ; H^1 (\Omega)^3)$ with $\dot{u}\in L^2 (0,T; L^2 (\Omega)^3)$. Moreover we even have that $u\in\mathcal{C}(0,T;H^1 (\Omega)^3)$ with $\dot{u}\in\mathcal{C}(0,T;L^2 (\Omega)^3)$ and for $u_0$, $u_1$ fixed the solution operator $L: L^2 (0,T; L^2 (\Omega)^3)\to  L^2 (0,T ; H^1 (\Omega)^3)$, $L(f) := u$ is continuous. If $u_0=u_1=0$, then $L$ even is a linear operator.\\
\end{proposition}

The Neumann condition \eqref{Ncond}, where $\nu$ is the outer unit normal vector field, means that the structure $\Omega$ is traction free at the boundary. A proof that is based on results from Lions \cite{LIONS:71} and the second inequality of Korn can be found in \cite{BINDER;SCHUSTER:15}. 
We introduce the notations $V:= H^1 (\Omega)^3$, $H:= L^2 (\Omega)^3$ and
\[  W^{1,q,r}(0,T;V,H) := \big\{ u\in L^q (0,T;V) : \dot{u}\in L^r (0,T;H)  \big\}  \]
for $1\leq q,r\leq +\infty$.

\vspace*{1ex}

A key ingredient to proof the uniform ill-posedness of the DLM-problem is the Lemma of Aubin-Lions, see \cite{LIONS:69}.

\vspace*{1ex}

\begin{theorem}[Lemma of Aubin-Lions]\label{L-Aubin-Lions}
If $q<+\infty$, then the embedding 
$$W^{1,q,r}(0,T;V,H) \hookrightarrow L^q (0,T;H)$$
is compact.
\end{theorem}

\vspace*{1ex}

In Theorem \ref{L-Aubin-Lions} we interpret $H$ as a subset of the dual space $V^*$.

\vspace*{1ex}

Let functions $\chi_j\in H^{1/2}(\partial\Omega)^3$, $j=,1,\ldots,J$, with small supports on $\partial\Omega$ be given which define the sensor characteristics, e.g., regarding size, sensitivity, etc. Then, a mathematical model for DLM is represented by the forward operator
\begin{equation}\label{F-DLM}
F[f] := S(t) \gamma \iota L[f],
\end{equation}
where $L: L^2 (0,T;H) \to W^{1,2,2}(0,T;V,H)$, $L[f]:=u$ maps the dynamic load to the unique weak solution of the IBVP
\eqref{elastic-WE-2}, \eqref{Ncond}--\eqref{IC2}, $\iota : W^{1,2,2}(0,T;V,H) \hookrightarrow L^2 (0,T;H)$ is the embedding
being compact due to the Lemma of Aubin-Lions, $\gamma: L^2 (0,T;H)\to L^2 (0,T;H^{-1/2}(\partial\Omega)^3)$ is the
trace operator, which is continuous, and 
\begin{equation}\label{observation}  
     S(t)[g] := \int_{\partial\Omega} \langle g(t,x),\chi_j (x)\rangle\,\mathrm{d}s_x,
     \qquad g\in L^2 (0,T;H^{-1/2}(\partial\Omega)^3),\; j=,1\ldots,J, \end{equation}
is the observation operator being linear and continuous as a mapping from 
\[ L^2 (0,T;H^{-1/2}(\partial\Omega)^3)\to L^2 (0,T;\RR^J). \] 
In \eqref{observation}, $\langle \cdot,\cdot \rangle$ is to be understood as the dual pairing in 
$H^{-1/2}(\partial \Omega)^3 \times H^{1/2}(\partial \Omega)^3$.
With the notations $\vartheta := f$, $K:= \gamma \iota L$, $\XX:= L^2(0,T;H)$ and $\YY:=L^2 (0,T;\RR^J)$
(i.e. $\mathcal{X}:=H$, $\mathcal{Y}:=\RR^J$) we immediately obtain the following result.

\vspace*{2ex}

\begin{theorem}
The operator $F = S(t) K : \XX \to \YY$ is compact.
\end{theorem}

\vspace*{2ex}

\begin{proof}
Since $\iota : W^{1,2,2}(0,T;V,H) \hookrightarrow L^2 (0,T;H)$ is compact, we have that $K$ is compact and so is $F$ as composition of a continuous and compact operator.
\end{proof}

\vspace*{2ex}

\begin{corollary}
The inverse problem of DLM, i.e., $(F,\XX,\YY)$, is uniformly ill-posed.
\end{corollary}

\vspace*{2ex}

\begin{proof}
This follows immediately from the compactness of $F$.
\end{proof}

\vspace*{2ex}

If we choose a fixed time $t_0$ and neglect time-dependence, then the DLM problem can be characterized by the elliptic problem
\begin{equation}\label{CE-static}
- \nabla\cdot (\mathbb{C}(x):\varepsilon (u))(x) = f(x),\qquad x\in\Omega,
\end{equation}
with traction-free boundary conditions
\begin{equation}\label{NCond-static}  [\mathbb{C}(x):\varepsilon (u)]\cdot \nu = 0\qquad \mbox{on } \partial\Omega .
\end{equation}
The weak formulation of the Neumann problem \eqref{CE-static}, \eqref{NCond-static} is given as
\begin{equation}\label{weak-formulation-static}
     \int_\Omega (\varepsilon (u),[\mathbb{C}(x):\varepsilon (v)])_F\,\mathrm{d} x = \int_\Omega f(x)\cdot v(x)\,
     \mathrm{d} x   
\end{equation}
for all $v\in V$. Using again \eqref{C}, the Poincar\'{e} inequality as well as the Lax-Milgram theorem, we have that
the equation $A(u) = f$ has a unique solution which depends continuously on $f$. Here, $A: V\to V^\ast$ is the operator
induced by the symmetric bilinear form in \eqref{weak-formulation-static}, i.e.
\[   A(u)[v] := \int_\Omega (\varepsilon (u),[\mathbb{C}(x):\varepsilon (v)])_F\,\mathrm{d} x, \quad v\in V.  \]
If we consider the inverse problem of computing the source term $f$ from full field data $u(x)$, $x\in\Omega$, then
we immediately get the following result.

\vspace*{2ex}

\begin{theorem}
The inverse problem $\widetilde{F}: H\to V$, $\widetilde{F}(f) := u$, where $u$ is the unique weak solution of 
\eqref{CE-static}, \eqref{NCond-static}, is well-posed.
\end{theorem}

\vspace*{2ex}

\begin{proof}
The well-posedness follows from the fact that $\widetilde{F}$ is continuously invertible as outlined above.
\end{proof}

\vspace*{2ex}

Of course, if we use the compact embedding $V \hookrightarrow H$, the trace operator $\gamma : L^2 (\Omega)^3
\to H^{-1/2} (\partial\Omega)$ and a (static) observation operator similar to $S$, then we obtain again
a linear, ill-posed (static) inverse problem.

\vspace*{2ex}

The examples above lead to the following general statements.

\vspace*{2ex}

\begin{proposition}
Let $V_1\subset \mathcal{X}\subset V_2$ be Banach spaces with compact embedding $V_1\subset \mathcal{X}$ and continuous embedding $\mathcal{X}\subset V_2$ Furthermore, let
$K: L^p (0,T;\mathcal{X}) \to W^{1,q,r}(0,T;V_1,V_2)$ with $1\leq q<+\infty$, $1\leq r\leq +\infty$, and $S(t): \mathcal{X}\to \mathcal{Y}$ is a family of uniformly bounded operators in $\mathcal{L}(\mathcal{X},\mathcal{Y})$. \\
Then $F : L^p (0,T;\mathcal{X}) \to L^q (0,T,\mathcal{Y})$, $F[\vartheta (t)] := S(t) K[\vartheta (t)]$ is compact and the dynamic inverse problem $(F,L^p (0,T;\mathcal{X}),L^q (0,T,\mathcal{Y}))$ is uniformly ill-posed. In case we have for any fixed $t=t_0$ that $K: \mathcal{X}\to V_1$ is continuous, then $\widetilde{F} : \mathcal{X} \to \mathcal{Y}$, $\widetilde{F}[\vartheta] := S (t_0) K[\vartheta]$, is compact and $(F,L^p (0,T;\mathcal{X}),L^q (0,T,\mathcal{Y}))$ is also temporally ill-posed. If for any $t_0$ the operator $S(t_0) K$ is continuously invertible, then
the problem $(F,L^p (0,T;\mathcal{X}),L^q (0,T,\mathcal{Y}))$ is temporally well-posed.
\end{proposition}

\vspace*{2ex}

\begin{proof}
The proof essentially relies on the compact embedding $V_1\subset \mathcal{X}$ as well as on the Lemma of Aubin-Lions which states that the embedding  $W^{1,q,r}(0,T;V_1,V_2)\hookrightarrow L^q (0,T;\mathcal{X})$ is compact.
\end{proof}

\vspace*{2ex}


It is now interesting to consider situations where the operator equation \eqref{ip} is temporally ill-posed but not uniformly ill-posed.
This is important to show that these are in fact different concepts with each of it having a justification of its own. As long as $K$ is
compact, $S(t)$ is linear and bounded for all $t\in [0,T]$ and the $f_n=F\vartheta_n$ are strongly measurable, parts (b) and (c) of the proof
will remain valid. Example \ref{E-nonuniform} shows a situation where the condition of uniform $L^q$-integrability of $\{f_n\}$ fails.

\vspace*{2ex}

\begin{example}\label{E-nonuniform}
We assume that the parameter $\vartheta$ to be recovered does not depend on time, i.e. $\vartheta (t) = \vartheta\in\X$ and use the embedding
$\iota : \X \hookrightarrow \XX=L^p (0,T;\X)$ which is defined by $x\mapsto f_x$ with $f_x (t) = x$. The range of this embedding consists of all
functions in $\XX$ that are constant in time and we write $\iota (X) =: L^p_c (0,T,\X) \subset L^p (0,T;\X)$. The forward operator $F:\iota(X)\subset\XX\to\YY$
is supposed to be given as 
\begin{equation}\label{Setting-constant}  [F \vartheta] (t) := S(t) K[\vartheta] ,\qquad \vartheta\in L_c^p(0,T;\X) \end{equation}
for a linear, compact operator $K:\X\to\Y$ and a family of linear, bounded mappings
$\{ S(t) :\Y\to\Y : t\in (0,T)\}$. Setting \eqref{Setting-constant} represents the important situation that we have to reconstruct a parameter that has only a spatial
variable from time-dependent data. Such a situation is, e.g., given in seismology, see \cite{KIRSCH;RIEDER:14}, where wave speed and mass density are computed from the full waveform. As a simple example we define
$S(t) = \frac{1}{t} I$ yielding $[F \vartheta] = \frac{1}{t} K[\vartheta]$ for $\vartheta\in L_c^p(0,T,\X)$, $t\in (0,T)$, and compact $K$.
Let $\{\vartheta_n\}_{n\in\NN}$ be a bounded sequence in $L_c^p(0,T,\X)$. Then, the sequence $\{f_n:= F\vartheta_n\}_{n\in\NN}$ is not
uniformly $L^q$-integrable and hence $F:\iota(X)\subset\XX\to\YY$ is not compact. This follows immediately from 
\begin{eqnarray*}
\lim_{k\to\infty} \sup_{n\in\NN} \int_{\{t:\|f_n(t)\|\geq k\}} \|f_n (t)\|^q\, \dd t &\geq &
\lim_{k\to\infty} \int_{\{t:\|f_{n^*}(t)\|\geq k\}} \|f_{n^*} (t)\|^q\, \dd t\\
&=& \lim_{k\to\infty} \|K[\vartheta_{n^*}]\|^q \int_0^{\|K[\vartheta_{n^*}]\|/k} \frac{1}{t^q}\, \dd t = +\infty
\end{eqnarray*}
for some $n^*\in\NN$ fixed. This means that $F \vartheta=y$ is not uniformly ill-posed provided that $K$ has an inverse which is bounded on $K(\iota(X))\subset \YY$.
But obviously, for fixed $t\in (0,T)$, $S(t) K[\vartheta]$ is compact as an operator from $\X$ to $\Y$ and hence $F \vartheta = y$
is temporally ill-posed.
\end{example}

\vspace*{2ex}

As a consequence of the considerations in Example \ref{E-nonuniform} we obtain\\

\begin{corollary}
The subset $\iota (\X)\subset \XX$ is not compact.
\end{corollary}

\vspace*{2ex}

A simple application of Definition \eqref{eq-uniformly-lp-integrable} furthermore shows that $\iota (X)$ is not relatively compact in $\XX$.

\vspace*{2ex}



\section{Regularization of time-dependent inverse problems}

In this section we define problem adapted classes of regularization methods for dynamic inverse problems that address the two different sorts
of ill-posedness. Again we conisder the inverse problem \eqref{ip} and aim for a stable solution of
\begin{equation}\label{ipdelta}
F (\vartheta) = y^\delta,
\end{equation}
where $y^\delta\in \YY$ denotes a noise contaminated version of the exact data $y$, i.e.
\begin{equation}\label{noisy_data}
\|y-y^\delta\|_{\YY} < \delta
\end{equation}
for a (small) positive noise level $\delta>0$. We assume that
\begin{equation}\label{cond_noise}  y^\delta (t) \in F_t (\mathcal{D}(F_t))\qquad \mbox{for all } t\in [0,T]  \end{equation}
and that there exists a solution $\vartheta^+$ of \eqref{ip} with
\begin{equation}\label{cond_exact}  \vartheta^+ \in \big( C([0,T];\X )\cap \mathcal{D}(F) \big)  \end{equation}
which implies that also $y(t)$ is well-defined and
\[  y (t)\in F_t (\mathcal{D}(F_t))\qquad \mbox{for all } t\in [0,T] . \]
We note that condition \eqref{cond_noise} is not an essential restriction with respect to applications. Usually data are acquired
for discrete time instances $t_k\in [0,T]$, $k=0,1,\ldots$, only. Hence data $y^\delta (t)$ that are continuous in time, can be obtained
by simple interpolation, e.g. using piecewise linear spline functions. Condition \eqref{cond_exact} can be justified by the fact
that $C([0,T];\X )\cap \mathcal{D}(F)$ is dense in $\mathcal{D}(F)$ and the fact that in applications the temporal development of the exact solution
mostly is continuous in time (at least this is not an essential confinement).\\

\begin{definition}\label{def_reg_t}
A \emph{temporal (pointwise) regularization method} for \eqref{ipdelta} is a family of mappings $\widetilde{R} : [0,T] \times \Y \times [0,+\infty) \to \X$
that satisfies the following condition: For $x_{t,\alpha}^\delta := \widetilde{R} (t, y^\delta (t), \alpha)$ there exists a \emph{parameter choice}
$\alpha: [0,T]\times [0,+\infty) \times \Y \to [0,\bar{\alpha})$, $0<\bar{\alpha}\leq +\infty$, such that
\[   \lim_{\delta\to 0} \| \vartheta^+ (t) - x_{t,\alpha (t,\delta,y^\delta (t))}^\delta \|_{\X} = 0\qquad \mbox{for all } t\in [0,T].   \]
\end{definition}

Definition \ref{def_reg_t} reflects the fact that dynamic inverse problems can be solved by defining a partition $\Delta=\{0=t_0<t_1<\cdots < t_N=T\}$
of $[0,T]$ and using a stationary regularization method for $F_{t_k}:\mathcal{D}(F_{t_k})\subset \X\to\Y$ for each $t_k$. This procedure is called 
\emph{tracking}.

\vspace*{2ex}

\begin{remark}
Obviously a temporal regularization yields an element $x_{t,\alpha}^\delta\in\X$ for $t\in [0,T]$ fixed. A stable regularization of \eqref{ipdelta}, however, demands 
for a solution in the Lebesgue-Bochner space $\XX$. But this can easily we achieved by simple interpolation techniques. Again assume that we have $x_{t_k,\alpha}^\delta\in\X$
be given for $t_k\in \Delta$, $k=0,\ldots,N$. Define $I \{x_{t_k,\alpha}^\delta\}$ as the piecewise constant interpolation which is defined as 
\[ I \{x_{t_k,\alpha}^\delta\}  (t) =  x_{t_k,\alpha}^\delta\qquad \mbox{for } t\in [t_k, t_{k+1}),\qquad k=0,\ldots,N-1.\]
Because of
\begin{eqnarray*}  \int_0^T \| I \{ x_{t_k,\alpha}^\delta \}  (t) \|^p_{\X}\, \dd t &=& \sum_{k=0}^{N-1} \int_{t_{k}}^{t_{k-1}} \| x_{t_k,\alpha}^\delta \|^p_{\X}\, \dd t\\
&\leq & T \max\{ \| x_{t_k,\alpha}^\delta \|^p_{\X} : k=0,\ldots,N-1 \} < +\infty   
\end{eqnarray*}
we see that $I \{ x_{t_k,\alpha}^\delta\}\in\XX$. Of course other interpolation methods, such as piecewise linear interpolation, can be used to obtain solutions that are smooth with respect to $t$.  
We emphasize that we see temporal regularization not as a regularization method in the Lebesgue-Bochner spaces $\XX$, $\YY$, rather than as regularization in $\X$, $\Y$ for fixed $t\in [0,T]$,
which is the core idea of tracking methods.
\end{remark}

\vspace*{2ex}

In contrast to tracking, problem \eqref{ip} can also be solved uniformly in $t$.

\vspace*{2ex}

\begin{definition}\label{def_reg}
A \emph{full (uniform) regularization method} for \eqref{ipdelta} is a family of mappings $R : \YY\times [0,+\infty) \to \XX$ that satisfies
the following condition: For $x_\alpha^\delta := R(y^\delta,\alpha)$ there exists a \emph{parameter choice} $\alpha: \YY\times [0,+\infty) \to
[0,\bar{\alpha})$, $0<\bar{\alpha}\leq +\infty$, such that 
\[   \lim_{\delta \to 0} \|\vartheta^+ - x_{\alpha(y^\delta,\alpha)}^\delta\|_{\XX} = 0.  \]
\end{definition}


\subsection{Example 1: variational tracking by Tikhonov functionals}


We define the functional $\J_{t,\alpha}^\delta : \mathcal{D}(F_t)\subset \X\to \RR$ by
\begin{equation}\label{Tikhonov-t}  
\J_{t,\alpha}^\delta (x) := S_t \big( F_t (x), y^\delta (t)\big) + \alpha_t \Omega_t (x),\qquad t\in [0,T] , \quad x\in\X,
\end{equation}
where $S_t : \Y\times \Y \to [0,+\infty)$ is a functional defining the data fitting term, $\Omega_t : \X\to [0,+\infty]$
is a penalty term to stabilize the reconstruction process and $\alpha_t>0$ acts as the regularization parameter. Note that we allow for the data fitting term, the penalty term, and the regularization parameter to depend on $t$.

\vspace*{2ex}

Based on the well-established theory for regularization methods in Banach spaces, see 
\cite{BENNING;BURGER:18, kazimierski:tikhonov1, scherzerpoeschl:convergencerates, skhk12}, we formulate assumptions
on $F_t$, $\mathcal{D}(F_t)$ and $\Omega_t$ that the minimization of \eqref{Tikhonov-t} yields a temporal regularization
for the special case that 
\begin{equation}\label{S_t}  
S_t ( \tilde{y}, \tilde{y}^\delta ) := \frac{1}{r} \big\|  \tilde{y}_1 - \tilde{y}_2  \big\|^r_{\Y},\qquad \tilde{y}_1,\tilde{y}_2\in\Y  
\end{equation}
with $1<r<\infty$.\\

\begin{proposition}\label{P-Tikhonov-t}
Under the assumptions that $\X$, $\Y$ are reflexive Banach spaces, $C([0,T];\X )\cap \mathcal{D}(F)$ is dense in $\mathcal{D}(F)$, and that for
\emph{fixed} $t\in [0,T]$ we have that $F_t : \mathcal{D}(F_t)\subset \X\to \Y$ is weak-to-weak sequentially continuous,
$\Omega_t : \X\to [0,+\infty]$ is proper, convex and lower semi-continuous, $\mathcal{D}(F_t)\cap \mathcal{D}(\Omega_t)
\not= \emptyset$ and the level sets $\mathcal{M}_{\Omega_t} := \{ x\in\X : \Omega_t\leq c\}$ are weakly sequentially 
pre-compact. Furthermore we choose an index function $\alpha_t: [0,\infty)\to [0,\infty)$, i.e. $\alpha_t$ is strictly
increasing, continuous with $\alpha_t (0) = 0$, which has the asymptotic behavior
\[  \alpha_t (\delta) \to 0 \qquad \mbox{and}\qquad \frac{\delta^r}{\alpha_t (\delta)} \to 0\qquad \mbox{as}\qquad
    \delta\to 0 .  \]
Then the Tikhonov functional \eqref{Tikhonov-t} with data fitting term \eqref{S_t} has a minimizer $x_{t,\alpha}^\delta\in \mathcal{D}(F_t)$ and 
$\widetilde{R}(t,y^\delta (t),\alpha):=x_{t,\alpha}^\delta$ is a temporal regularization method in the sense that,
if $\{\delta_n\}\subset (0,\infty)$ is a sequence with $\delta_n\to 0$ as $n\to\infty$, the sequence
$\{x_{t,\alpha (\delta_n)}^{\delta_n}\}$ has a weakly converging subsequence whose weak limit is an $\Omega_t$-minimizing
solution $\vartheta^+ (t)\in \X$ of $F_t (\vartheta(t)) = y(t)$.\\
\end{proposition}

Since for fixed $t$ the operator $F_t [\vartheta(t)] := S(t) K[\vartheta (t)]$ from \eqref{F_bounded_compact} is linear and compact,
we immediately get\\

\begin{corollary}
If $C([0,T];\X )\cap \mathcal{D}(F)$ is dense in $\mathcal{D}(F)$ and $\Omega_t$ satisfies the assumptions in Proposition \ref{P-Tikhonov-t},
then the Tikhonov regularization \eqref{Tikhonov-t}, \eqref{S_t} yields a temporal regularization method for \eqref{F_bounded_compact} in
the sense of Proposition \ref{P-Tikhonov-t}.\\
\end{corollary}

\begin{remark}
Variational tracking as in Proposition \eqref{P-Tikhonov-t} naturally leads to \emph{dynamic algorithms} according to Osipov et al. \cite{OSIPOV;ET;AL:00},
i.e., if two data sets $y_1^\delta$, $y_2^\delta$ coincide at a given time interval, $y_1^\delta (t) = y_2^\delta (t)$ for all $t\in [0,t_0]$
for given $t_0\in (0,T]$, then the algorithm's output coincides on this time interval as well. In our setting it is quite obvious that
under this assumption we have $\widetilde{R}(t,y_1^\delta (t),\alpha) = \widetilde{R}(t,y_2^\delta (t),\alpha)$ for all $t\in [0,t_0]$.
The reason is that variational tracking just means to compute a temporal frame of stationary solutions. Dynamic algorithms in
the sense of Osipov et al. inherently show causality since, if $y_1^\delta (t)\not= y_2^\delta (t)$ for a $t>t_0$, this does not affect the
output in the interval $[0,t_0]$.
\end{remark}

\vspace*{2ex}



\subsection{Example 2: variational regularization on Lebesgue-Bochner spaces}


Temporal regularization completely neglects topology, regularity and geometry of the 
corresponding functional time-space. To obtain a holistic regularization it is more convenient to develop regularization methods for
\eqref{ipdelta} in $\XX$ and $\YY$. One possibility is to use variational regularization techniques. To this end we define the Tikhonov functional
$\mathcal{J}_\alpha^\delta : \mathcal{D}(F) \subset \XX\to \RR$ by
\begin{equation}\label{Tikhonov-uniform}
\mathcal{J}_\alpha^\delta (\vartheta) := S \big( F(\vartheta),y^\delta \big) + \alpha \Omega (\vartheta),\qquad \vartheta\in\XX,
\end{equation}
where, again, $S: \YY \times \YY\to [0,+\infty )$ is an error functional denoting the data fitting term and $\Omega : \XX \to [0,+\infty]$ is
a penalty term whose influence is controlled by the parameter $\alpha>0$. Of course the most popular choice for the data fitting term
is again a power of the norm residual
\begin{equation}\label{S-uniform}
S (y,y^\delta) := \frac{1}{r} \big\|y_1-y_2\big\|_{\YY}^r,\qquad y_1,y_2\in \YY
\end{equation}
with $1<r<\infty$.

\vspace*{2ex}

Accordingly, we can use established results on regularization theory (see again \cite{BENNING;BURGER:18, kazimierski:tikhonov1, scherzerpoeschl:convergencerates, skhk12}) to state the following result:

\vspace*{2ex}

\begin{proposition}\label{P-Tikhonov-uniform}
Let us assume that $\X$, $\Y$ are reflexive Banach spaces with $1<p<\infty$. Then the forward operator $F:\mathcal{D}(F)\subset \XX\to \YY$ is weak-to-weak
sequentially continuous, $\Omega : \XX\to [0,+\infty]$ is proper, convex and lower semi-continuous, $\mathcal{D}(F)\cap \mathcal{D}(\Omega)
\not= \emptyset$ and the level sets $\mathcal{M}_{\Omega} := \{ x\in\X : \Omega\leq c\}$ are weakly sequentially pre-compact.
Furthermore we choose an index function $\alpha: [0,\infty)\to [0,\infty)$, i.e. $\alpha$ is strictly
increasing, continuous with $\alpha (0) = 0$, that has the asymptotic behavior
\begin{equation}\label{alpha-a-priori} 
\alpha (\delta) \to 0 \qquad \mbox{and}\qquad \frac{\delta^r}{\alpha (\delta)} \to 0\qquad \mbox{as}\qquad
\delta\to 0 .  \end{equation}
Then the Tikhonov functional \eqref{Tikhonov-uniform} with data fitting term \eqref{S-uniform} has a minimizer $x_{\alpha}^\delta\in \mathcal{D}(F)$ and 
$R(y^\delta,\alpha):=x_{\alpha}^\delta$ is a full (uniformly) regularization method in the sense that,
if $\{\delta_n\}\subset (0,\infty)$ is a sequence with $\delta_n\to 0$ as $n\to\infty$, then the sequence
$\{x_{\alpha (\delta_n)}^{\delta_n}\}$ has a weakly converging subsequence whose weak limit is a $\Omega$-minimizing
solution $\vartheta^+ \in \XX$ of $F (\vartheta) = y$.\\
\end{proposition}

It can be shown, that for the specific setting \eqref{F_bounded_compact} and $\Omega (\vartheta)$ being a power of the norm in $\XX$, the Tikhonov
method \eqref{Tikhonov-uniform} with data fitting term \eqref{S-uniform} yields a full regularization method for \eqref{ipdelta}.\\

\begin{theorem}\label{T-reg-bound-compact}
Let $\X$, $\Y$ are reflexive Banach spaces, $1<p<\infty$, and $F:\XX\to\YY$ be defined as in \eqref{F_bounded_compact}. Furthermore let the penalty
term $\Omega$ be defined as
\[   \Omega (\vartheta) := \frac{1}{q} \big\| \vartheta \|^q_{\XX},\qquad \vartheta\in\XX  . \]
Then $R:\YY\times (0,+\infty)\to \XX$, where $R(y^\delta,\alpha):=x_{\alpha}^\delta$ is the minimizer of the Tikhonov functional \eqref{Tikhonov-uniform} 
with data fitting term \eqref{S-uniform}, and a priori parameter choice $\alpha (\delta)$ as in \eqref{alpha-a-priori} is a full regularization method for \eqref{ipdelta} in the 
sense of Proposition \ref{P-Tikhonov-uniform}.\\
\end{theorem}

\begin{proof}
Since every bounded set in a Lebesgue-Bochner space $\XX$ has a weakly converging subsequence, it immediately follows that the level sets $\mathcal{M}_c (\Omega)$
are weakly sequentially pre-compact.\\
It remains to show that $F$ is weak-to-weak sequentially continuous. Let $\{\vartheta_n\}\subset \XX$ be a sequence with $\vartheta_n \rightharpoonup \vartheta$
weakly as $n\to\infty$ to a limit $\vartheta\in\XX$. Since $\XX^* \cong L^{p^*}(0,T;\X^*)$ and $\YY^* \cong L^{q^*}(0,T,\Y^*)$ we have for every $y^\ast\in
L^{q^*} (0,T;\Y^\ast)$ that
\begin{eqnarray*}
\langle y^*, F(\vartheta_n) \rangle_{\YY^*\times \YY} &=& \int_0^T \langle y^* (t), S(t) K[\vartheta_n (t)] \rangle_{\Y^*\times \Y}\dd t\\
&=&\int_0^T \langle K^* S(t)^* y^* (t),\vartheta_n (t) \rangle_{\X^*\times \X}\dd t
\end{eqnarray*} 
converges to
\begin{equation}\label{eq-weak-conv} 
\int_0^T \langle K^* S(t)^* y^* (t),\vartheta (t) \rangle_{\X^*\times \X} = \langle y^\ast,F(\vartheta) \rangle_{\YY^*\times \YY} 
\end{equation}
as $n\to\infty$ due to the weak convergence of $\vartheta_n \rightharpoonup \vartheta$. Equation \eqref{eq-weak-conv} proves that $K^* S(t)^* y^* (t)\in \XX^*$
and hence can be represented by a function from $L^{p^*}(0,T;\X^*)$.This shows the weak sequential continuity of $F$
and the assertion follows from Proposition \ref{P-Tikhonov-uniform}. 
\end{proof}


\subsection{Kaczmarz-based regularization for problems with static source}
We furthermore want to emphasize that inverse problems with time-dependent data and/or time-dependent forward operator are often formulated in a semi-discrete setting
\begin{equation} \label{eq:semi-discrete}
 F_i(x) = y_i, \quad i = 0,...,N-1,
\end{equation}
where the indices $i$ refer to discrete time points $t_i \in [0,T]$ or sections $I_i := [t_i,t_{i+1}]$ of the time interval $[0,T]$ at which the measurements are taken. The parameter that is to be identified is static in this setting. These problems are usually temporally or locally ill-posed and can be regularized using Kaczmarz's method, possibly in combination with other iterative methods such as the Landweber technique \cite{ttnn19} or sequential subspace optimization \cite{bhw20}. We introduce three scenarios that have been addressed in the literature and which fit into this framework.\\

\begin{example}
 We consider magnetic particle imaging. If the concentration $c$ of magnetic particles inside a body $\Omega$ is static, i.e., independent of time, the problem of reconstructing $c$ from measurements of the induced voltages $v_k$, $k=1,...,K$, is formulated as
 \begin{displaymath}
  v_k(t) = S_k(t) K_k[c](t), \quad k=1,\ldots ,K, 
 \end{displaymath}
 with
 \begin{displaymath}
  K_k[c](t) = \int_{\Omega} c(x) s_k(x,t) \,\mathrm{d}x
 \end{displaymath}
 and
 \begin{displaymath}
  S_k(t) = \int_{0}^T \widetilde{a}_k(t-\tau) \int_{\Omega} c(x) s_k(x,\tau) \mathrm{d}x \, \mathrm{d}\tau.
 \end{displaymath}
 The measurements are taken at time instances $t_i \in [0,T]$, $i=1,...,N$. We thus have a time-dependent forward operator 
 \begin{displaymath}
  F = \left(S_k K_k\right)_{k=1,...,K} \, : \, L^2(\Omega) \to L^2(0,T;\mathbb{R}^3)
 \end{displaymath}
 and time-dependent data $v \in L^2(0,T;\mathbb{R}^3)$.\\
\end{example}

\begin{example}
 In \cite{kaltenbacher17, ttnn19}, initial boundary value problems of the form
 \begin{align*}
  \partial_t u &= f(t,u(t),\vartheta) && \text{in } (0,T) \times \Omega, \\
             u &= 0                && \text{on } (0,T) \times \partial\Omega, \\
          u(0) &= u_0              && \text{in } \lbrace 0 \rbrace \times \Omega, \\
             y_i &= C_i u,          && i = 1,...,N,
 \end{align*}
 are considered, where the parameter $\vartheta \in \mathcal{X}$ is to be identified from measurements $y_i = C_i u := (C u(t_i))$ at time instances $t_i \in (0,T)$ of the state function 
 \begin{displaymath}
  u \in W^{1,p,p^*}(0,T;V,V^*) = \left\lbrace v \in L^p(0,T;V) \, : \, \partial_t v \in L^{p^*}(0,T;V^*) \right\rbrace \subseteq C(0,T; H).
 \end{displaymath}
 In particular, the parameter $\vartheta \in \mathcal{X}$ is assumed to be independent of time. In a Hilbert space setting ($p=p^*=2$), we may choose, e.g., $\mathcal{X} = L^2(\Omega)$. \\
 In \cite{ttnn19}, the case $f(t,u(t),\vartheta) = \Delta u + \Phi(u) + \vartheta$ with a nonlinear function $\Phi$ is addressed. The respective initial boundary value problems arise in several applications. For instance, the choice $\Phi(u) = u(1-u^2)$ is related to superconductivity, see \cite{ttnn19}.
\end{example}

\vspace*{2ex}

\begin{example}
 It is also possible to include the time-dependence, for example a motion or a deformation, in the mathematical model while the parameter that is to be reconstructed is considered static. This is the case in dynamic CT, where a known deformation of the investigated object can be incorporated in the forward operator, yielding a Radon transform along curves, see, e.g., \cite{bh2014}. 
\end{example}

\vspace*{2ex}

As already mentioned, the respective semi-discrete problems \eqref{eq:semi-discrete} 
can be solved iteratively, for example by a combination of Kaczmarz' method with the Landweber iteration, see, e.g., \cite{mhalos07, kaltenbacher17, ttnn19}. The iteration reads
\begin{displaymath}
 \vartheta_{n+1} = \vartheta_n - \big(F_{[n]}'(\vartheta_n)\big)^* \big(F_{[n]}(\vartheta_n) - y_{[n]}\big) 
\end{displaymath}
and is stopped at $n_*$ with $[n_*] = [0]$ when the adapted discrepancy principle
\begin{displaymath}
 \left\lVert F_{[n]}(\vartheta_n) - y_{[n]} \right\rVert \leq \tau_{[n]} \delta_{[n]}
\end{displaymath}
is fulfilled for all $n = n_* - N, ..., n_*-1$.
Instead of using only one search direction, it is also possible to use multiple search directions
\begin{displaymath}
 \vartheta_{n+1} = \vartheta_n - \sum_{k \in I_{n}} t_{n,k} \big(F_{[n]}'(\vartheta_{k})\big)^* \big(F_{[n]}(\vartheta_k) - y_{[n]}\big),
\end{displaymath}
which has been proposed in \cite{bhw20}. In this sense, one loop through all time-instances corresponds to one full iteration, in which the static source $\vartheta$ is reconstructed.




\section{Conclusion and outlook}

In this article we developed two novel concepts for ill-posedness and regularization of time-dependent inverse problems, i.e., the stable computation of time-dependent parameters from data varying in time. For such problems a mathematical setup in Lebesgue-Bochner spaces is convenient since, e.g., solutions of hyperbolic or parabolic PDEs show different regularities in time and space. But classical treatments of inverse problems usually rely on static Hilbert and Banach spaces. For both concepts, the pointwise (temporal) and uniform ill-posedness and regularization, we gave examples such as temporal observations of compact operators, variational tracking, Tikhonov and Kaczmarz-based methods. Future research will further extend these theoretical findings to more general Bochner spaces and aims at an integrated treatment of time-dependent inverse problems in the linear and nonlinear regimes. Last but not least the theoretical fundaments have to be supported by applications.


\bibliographystyle{siam}
\bibliography{bibliography_dip,references-hyperelastic}

\end{document}